\documentclass[english,11pt]{amsart}
\usepackage{amsopn,amsthm,amsfonts,amsmath,amssymb,amscd}
\usepackage{babel}
\usepackage{amssymb}
\usepackage{graphicx}

\newtheorem{Theorem}{Theorem}[section]
\newtheorem{Lemma}[Theorem]{Lemma}
\newtheorem{Proposition}[Theorem]{Proposition}
\newtheorem{Corollary}[Theorem]{Corollary}
\newtheorem{Example}[Theorem]{Example}

\newcommand{\DD}{\mathcal{D}}
\newcommand{\OO}{\mathcal{O}}

\newcommand{\cN}{\mathcal{N}}
\newcommand{\cA}{\mathcal{A}}
\newcommand{\cM}{{m}}

\def\Gl{\operatorname{GL}}
\def\Soc{\operatorname{Soc}}
\def\gr{\operatorname{gr}}

\begin{document}

\title{On pairs of commuting nilpotent matrices}
\author{Toma\v z Ko\v sir and Polona Oblak}
\date{\today}

\address{T. Ko\v sir:~Department of Mathematics, Faculty of Mathematics
and Physics, University of Ljubljana, Jadranska 19, SI-1000 Ljubljana, Slovenia; e-mail: 
tomaz.kosir@fmf.uni-lj.si.}
\address{P. Oblak: Department of Mathematics, Institute of Mathematics, Physics, and Mechanics, 
Jadranska 19, SI-1000 Ljubljana, Slovenia; e-mail: polona.oblak@fmf.uni-lj.si. \newline\noindent Current address: 
University of Ljubljana, Faculty of Computer and Information Science, Tr\v za\v ska cesta 25,
SI-1001 Ljubljana, Slovenia.}

\begin{abstract} 
Let $B$ be a nilpotent matrix and suppose that its Jordan canonical form is determined by a partition
$\lambda$. Then it is known that its nilpotent commutator $\cN_B$ is an irreducible 
variety and that there is a unique partition $\mu$ such that the intersection of the orbit of nilpotent matrices
corresponding to $\mu$ with 
$\cN_B$ is dense in $\cN_B$. We prove that map $\DD$ given by $\DD(\lambda)=\mu$ is an idempotent map. This answers a
question of Basili and Iarrobino \cite{Iarr2007} and gives a partial answer to a question of Panyushev \cite{Pany2007}.
In the proof, we use the fact that for a generic matrix $A\in\cN_B$ the algebra generated by $A$ and $B$ is
a Gorenstein algebra. Thus, a generic pair of commuting nilpotent matrices generates a Gorenstein algebra. 
We also describe $\DD(\lambda)$ in terms of $\lambda$ if $\DD(\lambda)$ has at most two parts.
\end{abstract}


\maketitle

\section{Introduction}

We denote by $M_n(F)$ the algebra of all $n\times n$ matrices over an algebraically closed field $F$ and by $\cN$
the variety of all nilpotent matrices in $M_n(F)$. 
Let $B\in \cN$ be a nilpotent matrix and suppose that its Jordan canonical form is determined by a partition
$\lambda$. We denote by $\OO_B=\OO_{\lambda}$ the orbit of $B$ under the $\Gl_n(F)$
action on $\cN$ and by $\cN_B$ the nilpotent commutator of $B$, i.e. the set
of all $A\in\cN$ such that $AB=BA$. It is known that $\cN_B$ is an irreducible variety (see Basili \cite{Bas2003}).
So there is a unique partition $\mu$ of $n$ such that $\OO_{\mu}\cap\cN_B$ is dense in $\cN_B$. Following Basili and
Iarrobino \cite{Iarr2007}, and Panyushev \cite{Pany2007} we define a map on the set of all partitions of $n$ by 
$\DD(\lambda)=\mu$. (We note that in \cite{BaIa2007,Iarr2007} this map is denoted by $Q$.)  
As in \cite{Pany2007} we say that $B$ or its orbit $\OO_{\lambda}$ is
\emph{self-large} if $\DD(\lambda)=\lambda$. According to \cite{BaIa2007,Iarr2007}, a partition $\lambda$ is called
\emph{stable} if $\DD(\lambda)=\lambda$.

The work presented in this paper was initially stimulated by a question if the partition $\DD(\lambda)$ is stable, i.e.
if $\DD$ is an idempotent map, that was posed by Basili and Iarrobino in their conference notes \cite{Iarr2007}. 
It was later brought to our attention that Panyushev in \cite[Problem 2]{Pany2007} stated the same question 
in a more general setup of simple Lie algebras. The general question in the theory of nilpotent orbits of semisimple 
Lie algebras is similar to the question described above. Suppose that $\mathfrak{g}$ is a semisimple Lie algebra and 
$G$ its adjoint group. If $x\in\mathfrak{g}$ is a nilpotent element
and $\mathfrak{z}_{\mathfrak{g}}(x)$ its centralizer in $\mathfrak{g}$, then there is
a unique maximal nilpotent $G$-orbit, say $Gy$, meeting ${\mathfrak
z}_{\mathfrak g}(x)$. It was pointed out to us by a referee that the latter follows from the fact that 
the set $\mathcal{N_{\mathfrak{h}}}$ of all nilpotent elements of an algebraic
Lie algebra $\mathfrak{h}$ is irreducible, which in turn can be deduced from the existence of the Levi decomposition and the
irreducibility of $\mathcal N_{\mathfrak{h}}$ in the reductive case, which was proved already by Kostant \cite[\S 3]{Kost1963}.
The question is then if the largest nilpotent orbit meeting ${\mathfrak z}_{\mathfrak g}(y)$ is $Gy$ itself 
\cite[Problem 2]{Pany2007}. A nilpotent orbit $Gy$ is said to be \emph{self-large}, if it is the
largest nilpotent orbit meeting ${\mathfrak z}_{\mathfrak g}(y)$. Thus, the question is if the largest nilpotent orbit
meeting ${\mathfrak z}_{\mathfrak g}(x)$ is necessarily self-large. 

Our main result is the proof of the fact that $\DD$ is an idempotent map on nilpotent orbits of $M_n(F)$. This answers 
the original question of Basili and Iarrobino and hence also a 
special ${\mathfrak sl}_n$ case of Panyushev's question. In the proof we use an extension of a lemma of 
Baranovsky \cite{Bar2001}.
We prove that a generic pair of commuting nilpotent matrices generates a Gorenstein algebra (in fact, a complete 
intersection \cite[Cor. 21.20]{Eis}). Moreover, a generic matrix $A\in \cN_B$ and $B$ generate
a Gorenstein algebra. Then
we use Macaulay's Theorem on the Hilbert function of the intersection of two plane curves \cite{Mac1904} 
(see also \cite{Iarr1994}) together with some results of \cite{BaIa2007}
to prove that $\DD$ is an idempotent. It appears that an answer to the general question posed 
by Panyushev \cite[Problem 2]{Pany2007} requires methods different from ours. We see no immediate generalization of 
our proof to the general setup of simple Lie algebras.

It is an interesting question \cite[Problem 1]{Pany2007} to describe $\DD(\lambda)$ in terms of partition $\lambda$. 
We would like to mention that some partial results 
to this question were obtained by the second author in \cite{Obl_indeks,Obl_schroer}. 
The main result of \cite{Obl_indeks} (see also \cite{BaIa_Invol}) gives the answer in the case when $\DD(\lambda)$ 
has at most two parts. We discuss this in the last section. 

Panyushev's work \cite{Pany2007} was motivated by Premet's results on the nilpotent 
commuting variety \cite{Prem_03} of a simple Lie algebra. For other results on special pairs of nilpotent commuting 
elements in these varieties see \cite{Ginz2000,Pany2001}. Some of other references for the theory of 
commuting varieties are \cite{Pany1994,PanYak07,Popo_08,Rich1979,Vasc1994}.

This paper is an extension of our unpublished note \cite{KosObl07}.

\vskip 20pt

\section{Gorenstein pairs are dense}

In this section we prove that a generic pair of commuting nilpotent matrices generates a Gorenstein (local 
artinian) algebra. 

Suppose that $B\in\cN$ and that $A\in\cN_B$. Then we denote by $\cA=\cA(A,B)$ the unital subalgebra
of $M_n(F)$ generated by matrices $A$ and $B$, and by $\cA^T=\cA(A^T,B^T)$ the unital subalgebra generated by the transposed 
matrices $A^T$ and $B^T$. The algebra $\cA$ is a commutative local artinian algebra. Such an algebra is Gorenstein if 
its socle $\Soc(\cA)$ is a simple $\cA$-module (see \cite[Prop. 21.5]{Eis}); i.e., if 
$\dim_F (\Soc(\cA))=1$, where $\Soc(\cA)=(0:\cM)$ is the annihilator of the maximal ideal $m$ of $\cA$.

We write $\cN_2\subset M_n(F)\times M_n(F)$ for the variety of all 
commuting pairs of nilpotent matrices. Note that the subset $U\subset \cN_2\times F^n\times F^n$
of all quadruples $(A,B,v,w)$ such that $v$ is a cyclic vector for $(A,B)\in\cN_2$ and $w$ cyclic for 
$(A^T,B^T)$ is an open subset. The fact that $U$ is open follows since its 
complement is given by a set of polynomial conditions $\det X = 0$ and $\det Y = 0$, where $X$ runs over all square 
matrices with columns of the form $A^iB^jv$ and $Y$ over all square matrices with the columns of the form 
$\left(A^T\right)^i\left(B^T\right)^jw$. Here it is certainly enough to take $0\le i,j \le n-1$. 
The same argument shows that also $U_B=\left\{(A,v,w);\ (A,B,v,w)\in U\right\}$ is an open subset 
of $\cN_B \times F^n \times F^n$.

\begin{Lemma}\label{lem1}
Suppose that $(A,B)\in\cN_2$. Then there is a third nilpotent matrix $C$ and vectors $v,w\in F^n$ such that:
\begin{enumerate}
\item[(i)] $C$ commutes with $B$,
\item[(ii)] any linear combination $\alpha A + \beta C$ is nilpotent,
\item[(iii)] $v$ is a cyclic vector for $(C,B)$,
\item[(iv)] $w$ is a cyclic vector for $(C^T,B^T)$.
\end{enumerate}
\end{Lemma}

\begin{proof} The proof is an extension of the proof of Baranovsky \cite[Lem. 3]{Bar2001}. Let 
$\lambda=\left(\lambda_1^{r_1},\lambda_2^{r_2}\ldots,\lambda_l^{r_l}\right)$, where $\lambda_1 > \lambda_1> \cdots >
\lambda_l>0$ and $r_i\ge 1$ for all $i$ be the partition
corresponding to $B$. As in the proof of \cite[Lem. 3]{Bar2001} there is a Jordan basis 
$$\left\{e_{ijk}: 1\le i \le l,\ 1\le j \le r_i,\ 1\le k 
\le \lambda_i\right\}$$ 
for $B$ such that
\begin{enumerate}
\item $Be_{ijk}=e_{ij,k+1}$ if $k < \lambda_i$ and $Be_{ij\lambda_i}=0$,
\item $Ae_{ijk}$ is in the linear span of vectors 
$e_{fgh}$, where either 
\begin{itemize}
\item $f>i$ and $g,h$ arbitrary, or 
\item $f=i$ and $g>j$ and $h$ arbitrary, or 
\item $f=i$ and $g=j$ and $h>k$.
\end{itemize}
\end{enumerate}
To simplify our expressions we assume that $e_{ijk}=0$ if the three indices  $i,j,k$ do not satisfy the conditions
$1\le i \le l,\ 1\le j \le r_i$ and $1\le k \le \lambda_i$. We also use the difference sequence of $\lambda$; we write $\delta_i=
\lambda_{i}-\lambda_{i-1}$ for $i=2,3\ldots,l$. Now we define $C$ by 
$Ce_{ijk} = e_{i,j+1,k}$ if $j < r_i$ and $Ce_{i,r_i,k} = e_{i+1,1,k} + e_{i-1,1,k+\delta_i}$. 

To illustrate the actions of $B$ and $C$ we include an example. We consider the case $\lambda=(4^2,3^2,2,1^2)$. 
To indicate the actions of $B$ and $C$ we draw a directed 
graph whose vertices correspond to the vectors of the Jordan basis 
$\{e_{ijk}\}$ and the edges to the nonzero elements of $B$ and $C$. In the graph in Figure \ref{fig}
 \begin{figure}[htb]
       \begin{center}
        \includegraphics[height=5cm, width=11cm]{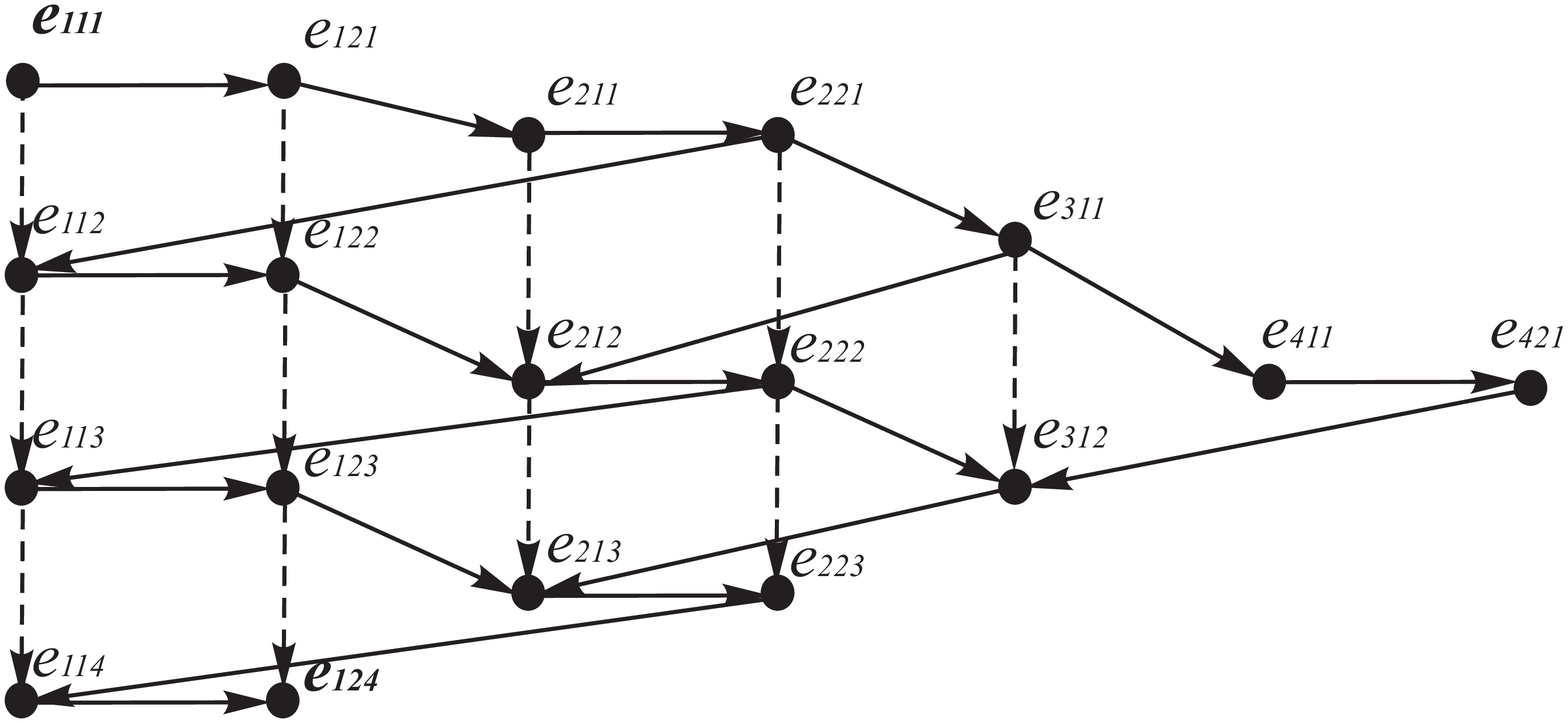}
       \end{center}
       \caption{}\label{fig}
 \end{figure}
the dashed directed edges correspond to the action of matrix $B$ and the nondashed ones correspond to the action of 
matrix $C$. Note that the graph indicating the action of $B^T$ and $C^T$ is obtained by reversing all the edges in the 
graph corresponding to $B$ and $C$. Note that vector $e_{111}$ is cyclic and vector $e_{124}$ cocyclic,
i.e., cyclic for $B^T$ and $C^T$.

Let us continue with the proof. It is easy to show that $BC=CB$; we prove directly that $BCe_{ijk}=CBe_{ijk}$ for 
all $i,j,k$. The property (ii) follows from (2) and the definition of $C$. 

Next we show that $e_{111}$ is a cyclic vector for $(C,B)$ and $e_{1r_1\lambda_1}$ is a cyclic vector for $(C^T,B^T)$. 
We denote by $\cA$ the unital subalgebra of $M_n(F)$ generated by $B$ and $C$ and by $\cA^T$ the unital subalgebra 
generated by the transposed matrices $B^T$ and $C^T$. We show by induction on $i$ that each $e_{ijk}$ is in the 
subspace $\cA e_{111}$. For $i=1$ we have $e_{1jk}=B^{k-1}C^{j-1}e_{111}$. We denote by $W_i$ the linear span of all 
vectors $e_{tjk}$ with $t\le i$. To prove the inductive step, we have to show that $e_{i+1,j,k}\in\cA W_i$. This 
follows since $e_{i+1,j,k}-B^{k-1}C^j e_{i,r_i,1}\in \cA W_i$. 

Before we prove that $e_{1r_1\lambda_1}$ is a cyclic vector for $(C^T,B^T)$ observe that $B^Te_{ijk}=e_{ij,k-1}$ and
that $C^Te_{ijk} = e_{i,j-1,k}$ if $j>1$ and $C^Te_{i1k} = e_{i-1,r_{i-1},k} + e_{i+1,r_{i+1},k-\delta_{i+1}}$. Again 
we proceed to prove that $e_{1r_1\lambda_1}$ is cyclic by induction on $i$.

For $i=1$ we have $e_{1jk}=\left(B^T\right)^{\lambda_1-k}\left(C^T\right)^{r_1-j}e_{1r_1\lambda_1}$. The inductive
step follows since
$e_{i+1,j,k}-\left(B^T\right)^{\lambda_i-k}\left(C^T\right)^{r_i-j}e_{ir_i\lambda_i}\in\cA^TW_i$.
\end{proof}

\begin{Proposition}\label{th1}
The subset $U$ is dense in $\cN_2\times F^n \times F^n$ and the subset $U_B$ is dense in
$\cN_B\times F^n \times F^n$.
\end{Proposition}

\begin{proof} Consider a quadruple $(A,B,v,w)\in\cN_2\times F^n \times F^n$. By Lemma \ref{lem1} we can find a matrix $C$
and vectors $v',w'$ such that $(C,B,v',w')\in U$. Then the affine line $L$ of all the points $(\alpha A + \alpha' C, B,
\alpha v + \alpha' v', \alpha w + \alpha' w')$, where $\alpha\in F$ is arbitrary and $\alpha'=1-\alpha$, has nonempty 
intersection with $U$. Hence $L\cap U$ is dense in $L$ and $(A,B,v,w)$ is in the closure of $U$.

The same argument shows that also $U_B$ is dense in $\cN_B\times F^n \times F^n$.
\end{proof}

Next we give a proof of a known result which we could not find stated in the literature. 

\begin{Proposition}\label{Gorenstein}
A commutative subalgebra $\mathcal{R}$ of $M_n(F)$ is Gorenstein if 
both $\mathcal{R}$ and $\mathcal{R}^T$ have a cyclic vector, i.e., the action of $\mathcal{R}$ is cyclic and cocyclic.
\end{Proposition}

\begin{proof} By Lemma 2.5, parts (1) and (2), of \cite{NeuSal} the subalgebra $\mathcal{R}$ is cyclic if and only if 
$F^n$ and $\mathcal{R}$ are isomorphic as $\mathcal{R}$-modules, and the subalgebra $\mathcal{R}^T$ is cyclic
if and only if $F^n$ and $\mathcal{R}^T$ are isomorphic $\mathcal{R}$-modules. Then our assumptions imply
that $\mathcal{R}$ and its dual module 
are isomorphic $\mathcal{R}$-modules, and thus $\mathcal{R}$ is Gorenstein by the definition 
\cite[pp. 525-526]{Eis}.
\end{proof}

We denote by $\pi:\cN_2\times F^n\times F^n\to\cN_2$ and $\pi_B:\cN_B\times F^n\times F^n\to\cN_B$ the projections 
to the first factor. Then, it follows by Proposition \ref{Gorenstein} that $\pi(U)$ is the set of 
pairs $(A,B)$ of nilpotent matrices such that the unital algebra $\cA$ generated by $A$ and $B$ is Gorenstein
of (vector space) dimension $n$. Moreover, such an $\cA$ is a complete intersection since its embedding dimension is 
at most two \cite[Cor. 21.20]{Eis}. Similarly, $\pi_B(U_B)$ is the set of all $A\in\cN_B$ such that $\cA(A,B)$
is Gorenstein of dimension $n$.

As a consequence of Proposition \ref{th1} we obtain the main results of this section:

\begin{Corollary}\label{cor01}
The subset $\pi(U)$ of those pairs in $\cN_2$ that generate a Gorenstein algebra of (vector space) dimension
$n$ is dense in $\cN_2$.
\end{Corollary}

\begin{Corollary}\label{cor1}
The subset $\pi_B(U_B)$ is dense in $\cN_B$. So, the unital subalgebra $\cA$ generated by a nilpotent matrix 
$B$ and a generic matrix $A$ in $\cN_B$ is Gorenstein; 
moreover, it is a complete intersection. 
Equivalently, both $\cA$ and $\cA^T$ have a cyclic vector, i.e., the action of $\cA$ is cyclic and cocyclic. 
\end{Corollary}


\vskip 20pt

\section{$\DD$ is an idempotent map}

A pair of commuting nilpotent matrices $(A,B)$ generates a (unital) local artinian algebra $\cA$. We denote its maximal
ideal by $\cM$. The associated graded algebra of $\cA$ is 
$$\gr \cA = \oplus_{i=0}^k \cM^i/\cM^{i+1},$$ 
where $\cM^0=\cA$ and
$k$ is such that $\cM^k\neq 0$ and $\cM^{k+1}=0$. 
The \emph{Hilbert function} $H(\cA)$ of $\cA$ is the sequence $(h_0,h_1,\ldots,h_k)$, where 
$$h_i=\dim_F \cM^i/\cM^{i+1}.$$ 
We have 
$\sum_{i=0}^k h_i =\dim_F\cA$. If $\cA$ is Gorenstein then $h_k=1$. Furthermore, Macaulay's Theorem 
on the Hilbert function of the intersection of two plane curves \cite{Mac1904} (see Iarrobino \cite{Iarr1977} or 
\cite[p. 23]{Iarr1994})
says that the Hilbert function $H(\cA )$ of an artinian local complete intersection of the embedding dimension at most $2$ 
satisfies 
$$H(\cA )=(1,2,\ldots,d, h_d,h_{d+1},\ldots,h_i,\ldots,h_k),$$ 
where $h_{d-1}=d\ge h_d\ge h_{d+1}\ge \ldots\ge 1$ and $h_{i-1}-h_{i}\le 1$ for all 
$i=d,d+1,\ldots,k$ and $h_k=1$. 

In the rest of the paper, we write $\lambda=(\lambda_1,\lambda_2,\ldots,\lambda_l)$, where $\lambda_1
\ge\lambda_2\ge\cdots\ge\lambda_l\ge 1$. (Compare with the proof of Lemma \ref{lem1},
where we used different, 'power type', notation for parts of partition $\lambda$.) We assume that $\lambda_j=0$ for $j>l$. 

The set of all partitions $\Lambda(n)$ of $n$ has a partial order given by
$\lambda\prec\mu$ if $\sum_{i=1}^j\lambda_i\leq\sum_{i=1}^j \mu_i$ for $j=1,2,\ldots$. To each partition $\lambda$
we associate its \emph{Ferrers diagram}, a diagram with $\lambda_i$ boxes in the $i$-th row. 
We denote by $\lambda(H)$ the 
partition, 
which has 
the $i$-th part 
equal to the number of elements in the sequence $H=(h_1,h_2,\ldots, h_k)$ such that $h_j\ge i$.

\begin{Example}
If $H=(1,2,3,2,1)$ then $\lambda(H)=(5,3,1)$ and if 
$H=(1,2,3,3,1)$ then $\lambda(H)=(5,3,2)$.
\end{Example}

Here we recall two results of Basili and Iarrobino \cite{BaIa2007} that we use in the proof below. They proved 
\cite[Thm. 2.21]{BaIa2007} that the Hilbert functions of the algebras $\cA(A,B)$, $A\in\cN_B$, determine $\DD(\lambda)$:
\begin{equation}\label{Dlambda}
\DD(\lambda)=\sup\{\lambda(H(\cA));\ \cA=\cA(A,B),\, A\in\cN_B,\ \text{such that}\ \dim\cA=n\}.
\end{equation}
They also proved \cite[Thm. 1.12]{BaIa2007} that if the parts in a partition $\lambda$ differ by at least two then 
$\lambda$ is stable, i. e. $\DD(\lambda)=\lambda$. 

The following is the main result of our paper.


\begin{Theorem}\label{thm2} Assume that $B$ is a nilpotent matrix and $\lambda$ the corresponding partition. Then the 
partition $\DD(\lambda)$ corresponding to the generic element $A\in\cN_B$ has decreasing parts differing by at least $2$
and $\DD(\DD(\lambda))=\DD(\lambda)$, i.e. $\DD$ is idempotent.
\end{Theorem}
 
\begin{proof} Recall that $\DD(\lambda)$ is given by (\ref{Dlambda}). Corollary \ref{cor1} implies 
that for a generic $A\in\cN_B$ the algebra 
$\cA=\cA(A,B)$ is a complete intersection (therefore Gorenstein) of dimension $n$. Since $\cN_B$ is irreducible
and since $\DD(\lambda)$ is the partition corresponding to the generic $A\in\cN_B$, 
it is enough to take in (\ref{Dlambda}) the supremum over all $\lambda(H(\cA))$, where 
$\cA$ is a complete intersection. The above stated Macaulay's Theorem on the Hilbert function of the intersection of two 
plane curves then implies that 
the parts in $\DD(\lambda)$ differ by at least two. By \cite[Thm. 1.12]{BaIa2007} or \cite[Cor 1.5]{Iarr2007} 
it follows that $\DD(\DD(\lambda))=\DD(\lambda)$. 
\end{proof}

\vskip 20pt

\section{Description of $\DD(\lambda)$ when it has at most two parts}

A partition $\lambda$ is called a \emph{almost rectangular} if the largest and the smallest part of $\lambda$
differ by at most one. (Note that in \cite{BaIa2007} the term \emph{string} 
is used for almost rectangular 
partition.) We denote by $r_B$ or $r_{\lambda}$ the smallest number $r$ such that $\lambda$
is a union of $r$ almost rectangular partitions. (Here $B$ is a nilpotent matrix such that its Jordan canonical 
form is determined by $\lambda$.) Basili \cite[Prop. 2.4]{Bas2003} proved that a generic $A\in\cN_B$ has
rank $n-r_B$, which is maximal possible in $\cN_B$ and that $\DD(\lambda)$ has $r_B$ parts. 

\begin{Example}
For instance, if $\lambda=(4,3,2,1)$ and $\mu=(7,7,6,4,4,3,2)$ then $r_{\lambda}=2$ and $r_{\mu}=3$.
\end{Example}

It is an interesting and important question to describe $\DD(\lambda)$ in terms of $\lambda$ 
(see \cite[\S 3]{Pany2007}, in particular \cite[Problem 1]{Pany2007}). Using 
\cite[Thm. 13]{Obl_indeks} and  \cite[Cor. 3.29]{BaIa_Invol} it is easy to answer this question when 
$\DD(\lambda)$ has at most two parts, i. e. when $r_B\le 2$. Here we would like to remark that the 
proofs of Lemma 11 and Theorems 12 and 13 in \cite{Obl_indeks} hold over any field of characteristic $0$,
while Basili and Iarrobino assume in \cite{BaIa_Invol} that the underlying field is algebraically closed
of arbitrary characteristic.

Now, if $r_B=1$ then $\DD(\lambda)=(n)$. If $r_B=2$ then it is enough to know the maximal index of
nilpotency of an element of $\cN_B$ to describe $\DD(\lambda)$, since $\DD(\lambda)$ is the maximal partition
among those corresponding to elements $A\in\cN_B$ \cite[Lem. 1.6]{BaIa2007}. Applying
the description of the maximal index of nilpotency given in \cite[Thm. 13]{Obl_indeks}
and \cite[Cor. 3.29]{BaIa_Invol}, we have the following result.

\begin{Theorem}If $\lambda=(\lambda_1,\lambda_2,\ldots,\lambda_l)$, $\lambda_1\ge\lambda_2\ge\ldots\ge
\lambda_l\ge 1$ is such that $r_{\lambda}=2$ then
$\DD(\lambda)=(i_{\lambda},n-i_{\lambda})$, where $i_{\lambda}$ is the maximal index
of nilpotency in $\cN_B$ given by
$$i_{\lambda}= \max\limits_{1 \leq i < l} \left\{2(i-1)+\lambda_{i}+\lambda_{i+1}+\ldots+\lambda_{i+r}; \; 
\lambda_{i} - \lambda_{i+r}\leq 1,\ \lambda_{i-1}\ge 2\ {\rm if}\ i>1\right\} \, .$$
\end{Theorem}

\begin{Example}
Suppose that $\lambda=(4,4,3,3,2)$ and $\mu=(5,5,3,3,2)$. Then we have $\DD(\lambda)=(14,2)$ and $\DD(\mu)=(12,6)$.
\end{Example}

\vskip 20pt

\section*{Acknowledgement}

The authors would like to thank Anthony Iarrobino for stimulating comments that improved 
the quality of the paper. We would also like to thank him and Roberta Basili for sharing with us early versions of
their papers. We acknowledge financial support of the Research Agency of the Republic of Slovenia.

\end{document}